\documentclass[12pt]{article}
\usepackage{amssymb}
\usepackage{amsmath,amsthm}
\usepackage{cite}
\usepackage[dvips]{graphicx}
\usepackage{hyperref}
\usepackage{color}
\usepackage{setspace}
\usepackage{enumerate}
\usepackage{tkz-graph}
\usepackage{tikz}
\usepackage{algorithm}
\usepackage{algpseudocode}
\hypersetup{colorlinks=true}

\hypersetup{colorlinks=true, linkcolor=blue, citecolor=blue,urlcolor=blue}

\newtheorem{remark}{Remark}

 \newtheorem{lemma}[remark]{Lemma}
 \newtheorem{theorem}[remark]{Theorem}
 
 \newtheorem{corollary}[remark]{Corollary}

\algnewcommand\algorithmicforeach{\textbf{for each}}
\algdef{S}[FOR]{ForEach}[1]{\algorithmicforeach\ #1\ \algorithmicdo}

\title{On distances in generalized Sierpi\'{n}ski graphs}

\author{Alejandro Estrada-Moreno$^{(*)}$, Erick D. Rodr\'{i}guez-Bazan $^{(**)}$\\ Juan A. Rodr\'{\i}guez-Vel\'{a}zquez$^{(*)}$ 
\\
$^{(*)}${\small Departament d'Enginyeria Inform\`atica i Matem\`atiques,}\\
{\small Universitat Rovira i Virgili,}  {\small Av. Pa\"{\i}sos
Catalans 26, 43007 Tarragona, Spain.} \\{\small
alejandro.estrada\@@urv.cat, juanalberto.rodriguez\@@urv.cat} 
\\
$^{(**)}${\small Department of Matemathics,}\\
{\small Central University of Las Villas,}{ \small Carretera a Camajuan\'{i} km. $5\frac{1}{2}$. Villa Clara, Cuba.} \\{\small
erickrodriguezbazan\@@gmail.com}
}

\begin{document}
\maketitle

\begin{abstract}
In this paper we propose formulas for the distance between vertices of a generalized Sierpi\'{n}ski graph $S(G,t)$ in terms of the distance between vertices of the base graph $G$. In particular, we deduce a recursive formula for the distance 
between an arbitrary vertex and an  extreme vertex of  $S(G,t)$, and we obtain a recursive formula for the distance between two arbitrary vertices of $S(G,t)$  when the base graph is triangle-free.
From these recursive formulas, we provide algorithms to compute the distance between vertices of $S(G,t)$.
In addition,  we give an explicit formula for the diameter and radius of $S(G,t)$ when the base graph is a tree.  
\end{abstract}

\section{Introduction}

Let $G=(V,E)$ be a non-empty graph, and $t$ a positive integer. We denote by $V^t $ the set of words of length $t$ on alphabet $V $. The letters of a word $u$ of length $t$ are denoted by $u_1u_2...u_t$. The concatenation of two words $u$ and $v$  is denoted by $uv$. Klav\v{z}ar and Milutinovi\'c introduced in \cite{Klavzar1997} the graph  $S(K_n, t)$, $t\ge 1$,  whose vertex set is $V^t$, where
$\{u,v\}$ is an edge if and only if there exists $i\in \{1,...,t\}$ such that:
$$ \mbox{ (i) }  u_j=v_j, \mbox{ if } j<i; \mbox{ (ii) } u_i\ne v_i; \mbox{ (iii) } u_j=v_i \mbox{ and } v_j=u_i  \mbox{ if } j>i.$$
As noted in \cite{Hinz2013},  in a compact form, the edge set can be described as
$$\{\{wu_iu_j^{d-1},wu_ju_i^{d-1}\}:\,  u_i,u_j\in V, i\ne j; d\in [t]; w\in V^{t-d} \}.$$
The graph $S(K_3,t)$ is isomorphic to 
the graph of the  Tower of Hanoi with $t$ disks \cite{Klavzar1997}. Later, those graphs have been
called Sierpi\'{n}ski graphs in \cite{Klavzar2002} and they were studied by now from numerous points of view.
For instance,  the authors of \cite{Gravier...Parreau} studied identifying codes, locating-dominating codes, and total-dominating codes in Sierpi\'{n}ski graphs. In \cite{Hirtz-Holz} the authors  propose an algorithm, which makes use of three automata and the fact that there are at most two internally vertex-disjoint shortest paths between any two vertices, to determine all shortest paths in Sierpi\'{n}ski graphs. The authors of \cite{Klavzar2002} proved that for any $n\ge 1$ and $t\ge 1$, the Sierpi\'{n}ski graph $S(K_n,t)$ has a unique 1-perfect code (or efficient dominating set) if $t$ is even, and $S(K_n,t)$ has exactly $n$ 1-perfect codes if $t$ is odd. 
The Hamming dimension of a graph $G$ was introduced in \cite{Klavsar-Peterin} as the largest dimension of a Hamming graph into which $G$ embeds as an irredundant induced subgraph. 
That paper gives an upper bound for the Hamming dimension of the Sierpi\'{n}ski graphs $S(K_n,t)$ for $n\ge 3$. It also shows that the Hamming dimension of $S(K_3,t)$ grows as $3^{t-3}$. 
The idea of almost-extreme vertex of $S(K_n,t)$ was  introduced in \cite{Klavsar-Zeljic} as a
vertex that is either adjacent to an extreme vertex of $S(K_n,t)$
or is incident to
an edge between two subgraphs of $S(K_n,t)$
isomorphic to $S(K_n,t-1)$.  The authors of \cite{Klavsar-Zeljic} deduced explicit formulas  for the distance in $S(K_n,t)$
between an arbitrary vertex and an almost-extreme vertex.
Also they gave  a formula of the metric dimension of a Sierpi\'{n}ski graph, which was independently obtained by Parreau in her Ph.D. thesis.
The set $S_u=\{v\in V(S(K_n,t)): \text{ there exist two shortest } u, v-\text{paths in } S(K_n,t)\}$, where $u$ is any almost-extreme vertex of $S(K_n,t)$, was completely determined in \cite{Xue2014}.
The eccentricity of an arbitrary vertex of Sierpi\'{n}ski graphs was studied in \cite{Hinz2012} where the main result  gives an expression for the average eccentricity of $S(K_n,t)$. For a general background  on Sierpi\'{n}ski graphs, the reader is
invited to read  the comprehensive survey \cite{Klavzar2016(survey)}  and references therein.

This construction was generalized in \cite{GeneralizedSierpinski} for any graph $G=(V,E)$, by defining the $t$-th \emph{generalized Sierpi\'{n}ski graph} of $G$, denoted by  $S(G,t)$,  as the graph with vertex set $V^t$ and edge set $\{\{wu_iu_j^{d-1},wu_ju_i^{d-1}\}:\,  \{u_i,u_j\}\in E; d\in [t]; w\in V^{t-d} \}.$ Figure \ref{FigureExDistance} shows the graph $S((K_2)^c+K_2,3)$.

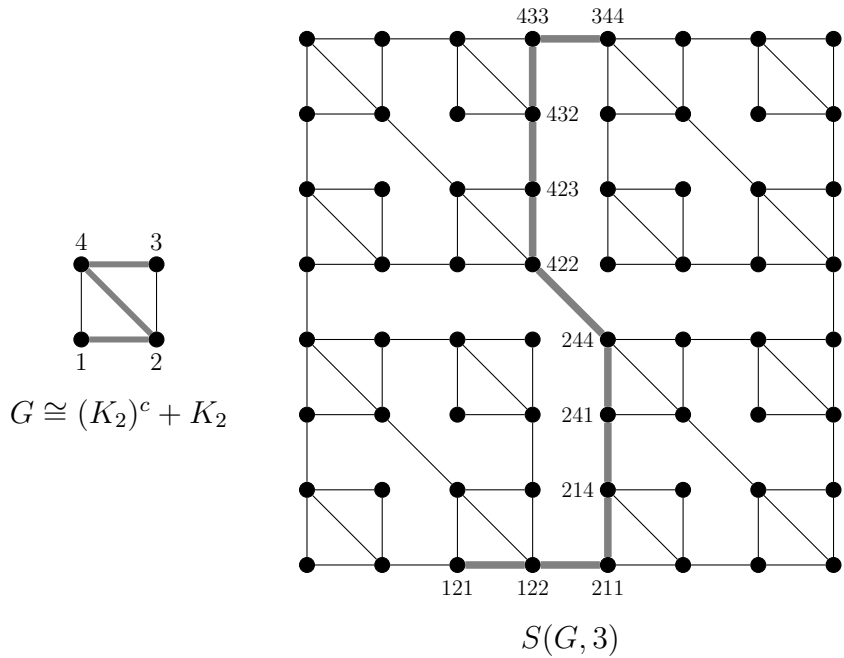
\begin{figure}[!ht]
\centering
\begin{tikzpicture}[transform shape, inner sep = .7mm]
%/////////////Graph G
\node [draw=black, shape=circle, fill=black] (1) at (0,0) {};
\node [scale=.8] at ([yshift=-.3 cm]1) {$1$};
\node [draw=black, shape=circle, fill=black] (2) at (1,0) {};
\node [scale=.8] at ([yshift=-.3 cm]2) {$2$};
\node [draw=black, shape=circle, fill=black] (3) at (1,1) {};
\node [scale=.8] at ([yshift=.3 cm]3) {$3$};
\node [draw=black, shape=circle, fill=black] (4) at (0,1) {};
\node [scale=.8] at ([yshift=.3 cm]4) {$4$};
\draw (2)--(1)--(4)--(2)--(3)--(4);
\draw[gray, line width=0.8 mm] (1)--(2)--(4)--(3);
\node at ([shift=({(.5,-1)})]1) {$G\cong (K_2)^c+K_2$};

%///////////Graph S(G,3)
\pgfmathsetmacro\traslationXThree{3};
\pgfmathsetmacro\traslationYThree{-3};
\foreach \lThree in {1,...,4}
{
\ifthenelse{\lThree=1}
{
\pgfmathsetmacro\traslationXTwo{\traslationXThree};
\pgfmathsetmacro\traslationYTwo{\traslationYThree};
}
{
\ifthenelse{\lThree=2}
{
\pgfmathsetmacro\traslationXTwo{\traslationXThree+4};
\pgfmathsetmacro\traslationYTwo{\traslationYThree};
}
{
\ifthenelse{\lThree=3}
{
\pgfmathsetmacro\traslationXTwo{\traslationXThree+4};
\pgfmathsetmacro\traslationYTwo{\traslationYThree+4};
}
{
\pgfmathsetmacro\traslationXTwo{\traslationXThree};
\pgfmathsetmacro\traslationYTwo{\traslationYThree+4};
};
};
};

%///////////Copy of graph S(G,2)

\foreach \lTwo in {1,...,4}
{
\ifthenelse{\lTwo=1}
{
\pgfmathsetmacro\x{\traslationXTwo};
\pgfmathsetmacro\y{\traslationYTwo};
}
{
\ifthenelse{\lTwo=2}
{
\pgfmathsetmacro\x{\traslationXTwo+2};

\pgfmathsetmacro\y{\traslationYTwo};
}
{
\ifthenelse{\lTwo=3}
{
\pgfmathsetmacro\x{\traslationXTwo+2};
\pgfmathsetmacro\y{\traslationYTwo+2};
}
{
\pgfmathsetmacro\x{\traslationXTwo};
\pgfmathsetmacro\y{\traslationYTwo+2};
};
};
};
\node [draw=black, shape=circle, fill=black] (\lThree\lTwo 1) at (\x,\y) {};
\node [draw=black, shape=circle, fill=black] (\lThree\lTwo 2) at (\x+1,\y) {};
\node [draw=black, shape=circle, fill=black] (\lThree\lTwo 3) at (\x+1,\y+1) {};
\node [draw=black, shape=circle, fill=black] (\lThree\lTwo 4) at (\x,\y+1) {};
\draw (\lThree\lTwo 2)--(\lThree\lTwo 1)--(\lThree\lTwo 4)--(\lThree\lTwo 2)--(\lThree\lTwo 3)--(\lThree\lTwo 4);
}
%Edges of each copy of S(G,2)
\foreach \u/\v in {1/2,2/3,3/4,1/4,2/4}
{
\draw (\lThree\u\v) -- (\lThree\v\u);
}
}
%Edges of S(G,3)
\foreach \u/\v in {1/2,2/3,3/4,1/4,2/4}
{
\draw (\u\v\v) -- (\v\u\u);
}
\node at ([shift=({(3.5,-1)})]111) {$S(G,3)$};
%Labels of vertices in the path
\node [scale=.7] at ([yshift=-.3 cm]121) {$121$};
\node [scale=.7] at ([yshift=-.3 cm]122) {$122$};
\node [scale=.7] at ([yshift=-.3 cm]211) {$211$};
\node [scale=.7] at ([xshift=-.4 cm]214) {$214$};
\node [scale=.7] at ([xshift=-.4 cm]241) {$241$};
\node [scale=.7] at ([xshift=-.4 cm]244) {$244$};
\node [scale=.7] at ([xshift=.4 cm]422) {$422$};
\node [scale=.7] at ([xshift=.4 cm]423) {$423$};
\node [scale=.7] at ([xshift=.4 cm]432) {$432$};
\node [scale=.7] at ([yshift=.3 cm]433) {$433$};
\node [scale=.7] at ([yshift=.3 cm]344) {$344$};
%Path with highlighted edges
\draw[gray, line width=1 mm] (121)--(122)--(211)--(214)--(241)--(244)--(422)--(423)--(432)--(433)--(344);
\end{tikzpicture}
\caption{The distance between vertices $121$ and $344$ in $S(G,3)$ is $10$.}
\label{FigureExDistance}
\end{figure}

Notice that if $\{u,v\}$ is an edge of $S(G,t)$, there is an edge $\{x,y\}$ of $G$ and a word $w$ such that $u=wxyy\dots y$ and $v=wyxx\dots x$. In general, $S(G,t)$ can be constructed recursively from $G$ with the following process: $S(G,1)=G$ and, for $t\ge 2$, we copy $n$ times $S(G, t-1)$ and add the letter $x$ at the beginning of each label of the vertices belonging to  the copy of $S(G,t-1)$ corresponding to $x$. Then for every edge $\{x,y\}$ of $G$, add an edge between vertex $xyy\dots y$ and vertex $yxx\dots x$. Vertices of the form $xx\dots x$ are called \textit{extreme vertices} of $S(G,t)$. Notice that for any graph $G$ of order $n$ and any integer $t\ge 2$,  $S(G,t)$  has $n$ extreme vertices and, if $x$ has degree $d(x)$ in $G$, then the extreme vertex $xx\dots x$ of $S(G,t)$   also has degree  $d(x)$. Moreover,   the degrees of two vertices $yxx\dots x$ and  $xyy\dots y$, which connect two copies of $S(G,t-1)$, are  equal to  $d(x)+1$ and $d(y)+1$, respectively. 

The authors of  \cite{GeneralizedSierpinski} announced some results about generalized Sierpi\'{n}ski graphs concerning their automorphism groups  and perfect codes. In our opinion, these results definitely deserve
to be published.
In the first published article  on
this subject   the authors obtained closed formulae for the Randi\'{c} index of polymeric networks modelled by  generalized Sierpi\'{n}ski graphs \cite{Rodriguez-Velazquez2015}, while in \cite{2015arXiv151007982E}  this work was extended to the so-called generalized Randi\'{c} index. Later, the total chromatic number of generalized Sierpi\'{n}ski graphs was  studied  in \cite{Geetha2015} and the strong metric dimension has recently been  studied in \cite{Rodriguez-Velazquez2016}.
The authors of \cite{Rodriguez-Velazquez2015a} obtained closed formulae for the chromatic, vertex cover, clique and domination numbers of generalized Sierpi\'{n}ski graphs $S(G,t)$ in terms of parameters of the base graph $G$. More recently,   a general upper bound on the Roman domination number of $S(G,t)$ was obtained in \cite{Ramezani2016}. In particular, it was studied the case in which the base graph $G$ is a path, a cycle, a complete graph or a graph having exactly one universal vertex.
In this paper we propose formulas for the distance between vertices of a generalized Sierpi\'{n}ski graph  in terms of the distance between vertices of the base graph. In particular, we deduce a recursive formula for the distance 
between an arbitrary vertex and an  extreme vertex of  $S(G,t)$, and we obtain a recursive formula for the distance between two arbitrary vertices of $S(G,t)$  when the base graph is triangle-free.
From these recursive formulas, we provide algorithms to compute the distance between vertices of $S(G,t)$.
In addition,  we give an explicit formula for the diameter and radius of $S(G,t)$ when the base graph is a tree.

\section{Distance 
between an arbitrary vertex and an  extreme vertex}

For any $t\ge 2$  the subgraph $\langle V_x \rangle$ of $S(G,t)$, induced by $V_x=\{xw:\; x\in V,w\in V^{t-1}\}$, is isomorphic to $
S(G,t-1)$. Note that $\langle V_x\rangle$ contains exactly one extreme vertex of $S(G,t)$. 
%%%%%%%
\begin{lemma}\label{lemmaDGInside}
Let $G=(V,E)$ be a connected non-trivial graph. For any  $x\in V $, $w,w'\in V^{t-1}$ and any integer $t\ge 2$, $$d_{S(G,t)}(xw, xw')=d_G(w,w').$$
\end{lemma}
\begin{proof}
 For any shortest path $w,w_1,w_2,\dots, w_l,w'$ between $w$ and $w'$ in $S(G,t-1)$ and $x\in V$, we have a path  $xw,xw_1,xw_2,\dots, xw_l,xw'$ between $xw$ and $xw'$ in $S(G,t)$. Hence,  we can conclude that 
\begin{equation}\label{EqForContradiction}
 d_{S(G,t)}(xw, xw')\le d_{S(G,t-1)}(w, w').
\end{equation} 
Suppose that there exists a shortest path $P$  between $xw$ and $xw'$ of the form 
$$xw=v_0w_0^{(0)},v_0w_1^{(0)},\ldots,v_0w_{l_0}^{(0)},v_1w_0^{(1)},v_1w_1^{(1)},\ldots,v_1w_{l_1}^{(1)},\ldots,v_0w_{l_r}^{(r)}=xw'.$$ 
In such a case, $v_0w_0^{(r)}=v_0(v_{r-1})^{t-1}$, $v_{r-1}w_{l_{r-1}}^{(r-1)}=v_{r-1}(v_{0})^{t-1}$, $v_{i+1}w_0^{(i+1)}=v_{i+1}(v_i)^{t-1}$ and $v_iw_{l_{i}}^{(i)}=v_{i}(v_{i+1})^{t-1}$ for all $i\in\{0,1,\ldots,r-2\}$. Also,  the trail $x=v_0,v_1,v_2,\ldots,$ $v_{r-1},v_0=x$ associated to $P$ has length greater than zero. Since $P$ is a shortest path, we obtain that $v_iw_0^{(i)},v_iw_1^{(i)},\ldots,v_iw_{l_i}^{(i)}$ is a shortest path in $\langle V_{v_i} \rangle$, where $i\in \{0,\dots    , r\}$ and $v_r=v_0$,
 so that 
\begin{description}
\item $d_{S(G,t)}\left(v_0w,v_0(v_1)^{t-1}\right)=d_{S(G,t-1)}\left(w,(v_1)^{t-1}\right)$,
\item $d_{S(G,t)}\left(v_0(v_{r-1})^{t-1},v_0w'\right)=d_{S(G,t-1)}\left((v_{r-1})^{t-1},w'\right)$,
\item $d_{S(G,t)}\left(v_{r-1}(v_{r-2})^{t-1},v_{r-1}(v_0)^{t-1}\right)=d_{S(G,t-1)}\left((v_{r-2})^{t-1},(v_0)^{t-1}\right)$ and
\item $d_{S(G,t)}\left(v_{i+1}(v_i)^{t-1},v_{i+1}(v_{i+2})^{t-1}\right)=d_{S(G,t-1)}\left((v_i)^{t-1},(v_{i+2})^{t-1}\right)$,  $i\in\{0,1,\ldots,r-3\}$.
\end{description}
 Thus, 
\begin{align*}
d_{S(G,t)}(xw,xw')=&d_{S(G,t)}(v_0w,v_0(v_1)^{t-1})+\sum_{i=0}^{r-3}d_{S(G,t)}(v_{i+1}(v_i)^{t-1},v_{i+1}(v_{i+2})^{t-1})+\\
&+d_{S(G,t)}(v_{r-1}(v_{r-2})^{t-1},v_{r-1}(v_0)^{t-1})+d_{S(G,t)}(v_0(v_{r-1})^{t-1},v_0w')+r\\
=&d_{S(G,t-1)}(w,(v_1)^{t-1})+\sum_{i=0}^{r-3}d_{S(G,t-1)}((v_i)^{t-1},(v_{i+2})^{t-1})+\\
&+d_{S(G,t-1)}((v_{r-2})^{t-1},(v_0)^{t-1})+d_{S(G,t-1)}((v_{r-1})^{t-1},w')+r.
\end{align*}
Hence, if $r$ is even, then
\begin{align*}
d_{S(G,t)}(xw,xw')>&d_{S(G,t-1)}(w,(v_1)^{t-1})+\sum_{i=0}^{\frac{r-4}{2}}d_{S(G,t-1)}((v_{2i+1})^{t-1},(v_{2i+3})^{t-1})\\
&+d_{S(G,t-1)}((v_{r-1})^{t-1},w')\\
\ge& d_{S(G,t-1)}(w,w') \text{ (by triangle inequality)},
\end{align*}
which contradicts \eqref{EqForContradiction}. Now, if $r$ is odd, then 
\begin{align*}
d_{S(G,t)}(xw,xw')>&d_{S(G,t-1)}(w,(v_1)^{t-1})+\sum_{i=0}^{\frac{r-5}{2}}d_{S(G,t-1)}((v_{2i+1})^{t-1},(v_{2i+3})^{t-1})\\
&+d_{S(G,t-1)}((v_{r-2})^{t-1},(v_0)^{t-1})
\sum_{i=0}^{\frac{r-3}{2}}d_{S(G,t-1)}((v_{2i})^{t-1},(v_{2i+2})^{t-1})\\
&+d_{S(G,t-1)}((v_{r-1})^{t-1},w')\\
\ge & d_{S(G,t-1)}(w,w') \text{ (by triangle inequality)},
\end{align*}
which contradicts \eqref{EqForContradiction}. Therefore, the result follows.
\end{proof}
%%%%%%%%%%%%%
From the lemma above, we deduce the following remark.

\begin{remark}
Let $G=(V,E)$ be a connected non-trivial graph and let $r\ge 1$ and $t\ge 1$ be two integers. 
If
\begin{equation}
v_0w_0^{(0)},v_0w_1^{(0)},\ldots,v_0w_{l_0}^{(0)},v_1w_0^{(1)},v_1w_1^{(1)},\ldots,v_1w_{l_1}^{(1)},\ldots,v_rw_0^{(r)},v_rw_1^{(r)},\ldots,v_rw_{l_r}^{(r)},\tag{$\ast$}\label{path}
\end{equation}   
is a shortest path 
in $S(G,t)$, where $v_i\in V$ and $w_{j}^{(i)}\in V^{t-1}$, for  
 $i\in \{0,1,\ldots,r\}$ and $j\in\{0,1,\ldots,l_i\}$,  
 then  $v_0,v_1,\ldots,v_r$ is a path in $G$.
\end{remark}

From now on, we will refer to the path $v_0,v_1,\ldots,v_r$ in $G$ as the $G$-path associated to the shortest path (\ref{path}).  We say that a $G$-path $P$ is \emph{triangle-free} if for any set $S$ composed by three consecutive vertices of $P$, the subgraph of $G$ induced by $S$ is a path. For instance, the $G$-path associated to the shortest path $121,122,211,214,241,244,422,423,432,433,344$ shown in Figure \ref{FigureExDistance} is $1,2,4,3$. Notice that this $G$-path  is not triangle-free. 

\begin{lemma}\label{lemmaDGExtremeVerticesFreeTriangle}
Let $G=(V,E)$ be a connected non-trivial graph and let $t\ge 1$ be an integer such that $d_{S(G,t)}(u^t, v^t)=(2^t-1)d_G(u,v)$ for all  $u,v\in V$. 
Then for every $x,y\in V$ and $w,w'\in V^t$ the following assertions hold. 

\begin{enumerate}[{\rm (i)}]
\item The $G$-path associated to any shortest path between $x^{t+1}$ and $yw$ is triangle-free.

\item Let  $w\ne x^{t-1}$,  $w'\ne y^{t-1}$. If $d_G(x,y)\ge 2$, then  the $G$-path $x,v_1,v_2,\ldots,$ $v_{r-1},y$ associated to any shortest path between $xw$ and $yw'$ is triangle-free whenever $x\not\in N(v_2)$.
\end{enumerate}
\end{lemma}

\begin{proof}
Let $P$ be a shortest path between $x^{t+1}$ and $yw$ and let $x=v_0,v_1,v_2,\ldots,$ $v_{r-1},v_r=y$ be the $G$-path $P_1$ associated to $P$. If $P_1$ is triangle-free, then we are done. Now, suppose that $j$ is the minimum subscript  such that $v_j$ and $ v_{j+2}$ are adjacent and   consider the following cases:
\begin{enumerate}[(a)]
\item $j=0$. Suppose that $r=2$. In this case, the $G$-path is $x=v_0,v_1,v_2=y$,  so that 
\begin{align*}
d_{S(G,t+1)}(x^{t+1},yw)=&d_{S(G,t+1)}((v_0)^{t+1},v_0(v_1)^t)+d_{S(G,t+1)}(v_{1}(v_0)^t,v_{1}(v_{2})^t)\\
&+d_{S(G,t+1)}(v_2(v_{1})^t,v_2w)+2\\
=&d_{S(G,t)}((v_0)^t,(v_1)^t)+d_{S(G,t)}((v_0)^t,(v_2)^t)\\
&+d_{S(G,t)}((v_{1})^t,w)+2\text{ (by Lemma \ref{lemmaDGInside})}\\
\ge &d_{S(G,t)}((v_0)^t,w)+d_{S(G,t)}((v_0)^t,(v_{2})^t)+2\\
&\text{(by triangle inequality)}\\
> &d_{S(G,t)}((v_0)^t,w)+d_{S(G,t)}((v_0)^t,(v_{2})^t)+1\\
= &d_{S(G,t+1)}((v_0)^{t+1},v_0(v_2)^t)+1+d_{S(G,t+1)}(v_2(v_0)^t,v_2w)\\
&\text{ (by Lemma \ref{lemmaDGInside})}\\
\ge & d_{S(G,t+1)}(x^{t+1},yw) \text{ (by triangle inequality)},
\end{align*}
which is a contradiction.

Now, assume that $r\ge 3$ and let $$\displaystyle \alpha_j=\sum_{i=2}^{r-2} d_{S(G,t+1)}(v_{i+1}v_i^t,v_{i+1}v_{i+2}^t)+d_{S(G,t+1)}(v_r(v_{r-1})^t,v_rw)+r-2.$$ So,
\begin{align*}
d_{S(G,t+1)}(x^{t+1},yw)=&d_{S(G,t+1)}((v_0)^{t+1},v_0(v_1)^t)+d_{S(G,t+1)}(v_{1}(v_0)^t,v_{1}(v_{2})^t)\\
&+d_{S(G,t+1)}(v_2(v_{1})^t,v_2(v_3)^{t})+\alpha_j+2\\
=&d_{S(G,t)}((v_0)^t,(v_1)^t)+d_{S(G,t)}((v_0)^t,(v_{2})^t)\\
&+d_{S(G,t)}((v_{1})^t,(v_3)^{t})+\alpha_j+2\text{ (by Lemma \ref{lemmaDGInside})}\\
\ge& d_{S(G,t)}((v_0)^t,(v_3)^t)+d_{S(G,t)}((v_0)^t,(v_2)^t)+\alpha_j+2\\
&\text{ (by triangle inequality)}\\
=& d_{S(G,t+1)}((v_0)^{t+1},v_0(v_2)^t)+d_{S(G,t+1)}(v_2(v_0)^t,v_2(v_3)^t)\\
&+\alpha_j+2 \text{ (by Lemma \ref{lemmaDGInside})}\\
>& d_{S(G,t+1)}((v_0)^{t+1},v_0(v_2)^t)+1+d_{S(G,t+1)}(v_2(v_0)^t,v_2(v_3)^t)\\
&+\alpha_j
\\
\ge & d_{S(G,t+1)}(x^{t+1},yw) \text{ (by triangle inequality)},
\end{align*}
which is a contradiction.  

\item $1\le j\le r-3$. Let $\displaystyle \alpha_j=d_{S(G,t+1)}((v_0)^{t+1},v_0(v_1)^t)+\sum_{i=0}^{j-2} d_{S(G,t+1)}(v_{i+1}v_i^t,v_{i+1}(v_{i+2})^t)+\sum_{i=j+2}^{r-2} d_{S(G,t+1)}(v_{i+1}v_i^t,v_{i+1}(v_{i+2})^t)+d_{S(G,t+1)}(v_r(v_{r-1})^t,v_rw)+r-2$. Notice that $d_G(v_{j-1},v_{j+1})=2,d_G(v_{j},v_{j+2})=1$ and $ d_G(v_{j+1},v_{j+3}),d_G(v_{j-1},v_{j+2}), d_G(v_{j},v_{j+3})\in \{1,2\}$. Hence, 
\begin{align*}
d_{S(G,t+1)}(x^{t+1},yw)=&\alpha_j+d_{S(G,t+1)}(v_{j}(v_{j-1})^t,v_{j}(v_{j+1})^t)\\
&+d_{S(G,t+1)}(v_{j+1}(v_{j})^t,v_{j+1}(v_{j+2})^t)\\
&+d_{S(G,t+1)}(v_{j+2}(v_{j+1})^t,v_{j+2}(v_{j+3})^t)+2\\
=&\alpha_j+d_{S(G,t)}((v_{j-1})^t,(v_{j+1})^t)+d_{S(G,t)}((v_{j})^t,(v_{j+2})^t)\\
&+d_{S(G,t)}((v_{j+1})^t,(v_{j+3})^t)+2\text{ (by Lemma \ref{lemmaDGInside})}\\
\ge& \alpha_j+2(2^t-1)+(2^t-1)+(2^t-1)+2 \text{ (by assumption)}\\
\ge &\alpha_j+d_{S(G,t)}((v_{j-1})^t,(v_{j+2})^t)+d_{S(G,t)}((v_{j})^t,(v_{j+3})^t)+2\\
& \text{ (by assumption)}\\
=&\alpha_j+d_{S(G,t+1)}(v_j(v_{j-1})^t,v_j(v_{j+2})^t)\\
&+d_{S(G,t+1)}(v_{j+2}(v_{j})^t,v_{j+2}(v_{j+3})^t)+2\text{ (by Lemma \ref{lemmaDGInside})}\\
>&\alpha_j+d_{S(G,t+1)}(v_j(v_{j-1})^t,v_j(v_{j+2})^t)+1\\
&+d_{S(G,t+1)}(v_{j+2}(v_{j})^t,v_{j+2}(v_{j+3})^t)
\\
\ge & d_{S(G,t+1)}(x^{t+1},yw) 
\text{ (by triangle inequality)},
\end{align*}
which is a contradiction.   
\item $j=r-2$ and $r\ge 3$. Let  $$\alpha_j=d_{S(G,t+1)}((v_0)^{t+1},v_0(v_1)^t)+\sum_{i=0}^{r-4} d_{S(G,t+1)}(v_{i+1}(v_i)^t,v_{i+1}(v_{i+2})^t)+r-2.$$ Notice that $d_G(v_{r-3},v_{r-1})=2,d_G(v_{r-2},v_{r-1})=d_G(v_{r-2},v_{r})=1$ and $ d_G(v_{r-3},v_{r})\in \{1,2\}$. In this case,
\begin{align*}
d_{S(G,t+1)}(x^{t+1},yw)=&\alpha_j+d_{S(G,t+1)}(v_{r-2}(v_{r-3})^t,v_{r-2}(v_{r-1})^t)\\
&+d_{S(G,t+1)}(v_{r-1}(v_{r-2})^t,v_{r-1}(v_{r})^t)\\
&+d_{S(G,t+1)}(v_r(v_{r-1})^t,v_rw)+2\\
=&\alpha_j+d_{S(G,t)}((v_{r-3})^t,(v_{r-1})^t)+d_{S(G,t)}((v_{r-2})^t,(v_{r})^t)\\
&+d_{S(G,t)}((v_{r-1})^t,w)+2 \text{  (by Lemma \ref{lemmaDGInside})}\\
= &\alpha_j+2(2^t-1)+(2^t-1)+d_{S(G,t)}((v_{r-1})^t,w)+2\\
& \text{ (by assumption)}\\
\ge &\alpha_j+d_{S(G,t)}((v_{r-3})^t,(v_{r})^t)+d_{S(G,t)}((v_{r-2})^t,(v_{r-1})^t)\\
&+d_{S(G,t)}((v_{r-1})^t,w)+2 \text{ (by assumption)}\\
> &\alpha_j+d_{S(G,t)}((v_{r-3})^t,(v_{r})^t)+d_{S(G,t)}((v_{r-2})^t,(v_{r-1})^t)\\
&+d_{S(G,t)}((v_{r-1})^t,w)+1\\
\ge &\alpha_j+d_{S(G,t)}((v_{r-3})^t,(v_{r})^t)+d_{S(G,t)}((v_{r-2})^t,w)+1\\
&\text{(by triangle inequality)}\\
= &\alpha_j+d_{S(G,t+1)}(v_{r-2}(v_{r-3})^t,v_{r-2}(v_{r})^t)+1\\
&+d_{S(G,t+1)}(v_r(v_{r-2})^t,v_rw)\text{ (by Lemma \ref{lemmaDGInside})}
\\
\ge & d_{S(G,t+1)}(x^{t+1},yw) \text{ (by triangle inequality)},
\end{align*}
which is a contradiction.
\end{enumerate}
Considering the previous cases, we deduce that the $G$-path associated to any shortest path between $x^{t+1}$ and $yw$ is triangle-free. Therefore, (i) holds. 

Now, assume that $d_G(x,y)\ge 2$. %Obviously, if  the $G$-path $x,v_1,v_2,\ldots,$ $v_{r-1},y$ associated to a shortest path between $xw$ and $yw'$ is triangle-free, then $x\not\in N(v_2)$ and $y\not\in N(v_{r-2})$. Now,
If $x\not\in N(v_2)$ and $x=v_0,v_1,v_2,\ldots,$ $v_{r-1},v_r=y$ is the $G$-path associated to a  shortest path between $xw$ and $yw'$, then by analogy to the proof of (i), cases (b) and (c), we deduce that the above mentioned $G$-path is triangle-free. Therefore, (ii) holds.
\end{proof}

\begin{theorem}\label{theoExtremeVertices}
Let $G=(V,E)$ be a connected non-trivial graph. For any $x,y\in V$ and any integer $t\ge 1$,  $$d_{S(G,t)}(x^t,y^t) = (2^t-1)d_G(x,y).$$
\end{theorem}

\begin{proof}
We will proceed by induction on $t$.  For $t = 1$, we have that $d_{S(G,1)}(x^1, y^1)=d_G(x,y)=(2^1-1)d_G(x,y)$. Suppose that $d_{S(G,t)}(x^t, y^t) = (2^t-1)d_G(x,y)$ holds true for an integer $t\ge 1$ and any pair of vertices of $G$. We will  show that $d_{S(G,t+1)}(x^{t+1}, y^{t+1}) = (2^{t+1}-1)d_G(x,y)$.

Let $P$ be a  shortest path between $x^{t+1},y^{t+1}$ and let $x=v_0,v_1,v_2,\ldots, v_{r-1},v_r=y$ be the $G$-path associated to $P$. So, 
\begin{align*}
d_{S(G,t+1)}(x^{t+1},y^{t+1}) =&d_{S(G,t+1)}(v_0^{t+1},v_0v_1^t)+\sum_{i=0}^{r-2} d_{S(G,t+1)}(v_{i+1}v_i^t,v_{i+1}v_{i+2}^t) +\\
&+d_{S(G,t+1)}(v_rv_{r-1}^t,v_r^{t+1})+r. 
\end{align*}
%\text{ (by Remark \ref{remarkPath})}
By  hypothesis and Lemmas \ref{lemmaDGInside} and \ref{lemmaDGExtremeVerticesFreeTriangle}, we obtain
$$d_{S(G,t+1)}(x^{t+1},y^{t+1}) = 2^t-1+2(2^t-1)(r-1)+2^t-1+r.
$$
Also, since $ r\ge d_G(x,y)$  we have $$d_{S(G,t+1)}(x^{t+1},y^{t+1})\ge (2^{t+1}-1)d_G(x,y).$$

Now, let $x=u_0,u_1,\dots,u_s=y$ be a shortest path between $x$ and $y$. By Lemma \ref{lemmaDGInside} and induction hypothesis we have that 
\begin{description}
\item $d_{S(G,t+1)}(x^{t+1}, x(u_1)^t)=2^t-1$, $d_{S(G,t+1)}(y(u_{s-1})^t, y^{t+1})=2^t-1$ and
\item $d_{S(G,t+1)}(u_{i+1}(u_i)^t, u_{i+1}(u_{i+2})^t)=2(2^t-1)$ for $i\in\{0,1,\ldots,s-2\}$.
\end{description}
Thus, since $u_i(u_{i+1})^{t}$ is adjacent to $u_{i+1}(u_i)^{t}$ for all $i\in\{0,1,\ldots,s-1\}$, we have 
\begin{align*}
d_{S(G,t+1)}(x^{t+1},y^{t+1})\le &d_{S(G,t+1)}(x^{t+1},x(u_1)^t)+\sum_{i=0}^{s-2} d_{S(G,t+1)}(u_{i+1}(u_i)^t,u_{i+1}(u_{i+2})^t)+\\
&+d_{S(G,t+1)}(y(u_{s-1})^t,y^{t+1})+s\\
=&2^t-1+\sum_{i=0}^{s-2} 2\left(2^t-1\right)+2^t-1+s\\
=&2^t-1+2\left(2^t-1\right)(d_G(x,y)-1)+2^t-1+d_G(x,y)\\
=& (2^{t+1}-1)d_G(x,y).
\end{align*}
Therefore, $d_{S(G,t+1)}(x^{t+1},y^{t+1}) =  (2^{t+1}-1)d_G(x,y).$
\end{proof}

For any $w\in V^{t-1}$, the subgraph induced by $\{wx:\, x\in V\}$ is isomorphic to $G$ and so $d_{S(G,t)}(x^t,x^{t-1}y)=d_G(x,y)$. Hence, as we will see in Theorem  \ref{remExtremeArbitraryVertices}, we only study  $d_{S(G,t)}(x^t,w)$ for the cases in which $w$ is not of the form $ x^{t-1}y$.

Given two vertices $x,y\in V$, we define $\mathcal{P}(x,y)$ as the set of all shortest paths between $x$ and $y$. For any $P_i\in \mathcal{P}(x,y)$, the neighbour of $y$ lying on $P_i$ will be denoted by $y^{(i)}$. With this notation in mind we can state the following result.

\begin{theorem}\label{remExtremeArbitraryVertices}
Let $G=(V,E)$ be a connected non-trivial graph. For any integer $t\ge 2$ and any $x\in V$ and $w=x^{j-1}z_jz_{j+1}\cdots z_t\in V^t$ such that $1\le j\le t-1$ and $x\ne z_j$, 
\begin{align*}
d_{S(G,t)}(x^t,w)=&\min_{P_i\in\mathcal{P}(x,z_j)}\left\{d_{S(G,t-j)}\left(\left(z_j^{(i)}\right)^{t-j},z_{j+1}\cdots z_t\right)\right\}+\\
&+(2^{t-j+1}-1)d_G(x,z_j)-(2^{t-j}-1).
\end{align*}
\end{theorem}

\begin{proof}
By Lemma \ref{lemmaDGInside}, we have that $d_{S(G,t)}(x^t,w)=d_{S(G,t-j+1)}(x^{t-j+1},z_j\cdots z_t)$. Let $x=v_0,v_1,v_2,\ldots, v_{r - 1},v_r=z_j$ be a shortest path $P$ between $x$ and $z_j$. Let $P_1$ be a path of minimum length among all the paths from $x^{t-j+1}$ to $z_j\cdots z_t$  having $P$ as its associated path. So, the length of $P_1$ is given by 
\begin{align}
l(P_1)=&d_{S(G,t-j+1)}(x^{t-j+1},x(v_1)^{t-j})+\sum_{i=0}^{r-2} d_{S(G,t-j+1)}(v_{i+1}(v_i)^{t-j},v_{i+1}(v_{i+2})^{t-j})+r\nonumber\\
&+d_{S(G,t-j+1)}(v_r(v_{r-1})^{t-j},z_j\cdots z_t)\nonumber\\
=&d_{S(G,t-j)}(x^{t-j},v_1^{t-j})+\sum_{i=0}^{r-2} d_{S(G,t-j)}((v_i)^{t-j},(v_{i+2})^{t-j})+r\nonumber\\
&+d_{S(G,t-j)}((v_{r-1})^{t-j},z_{j+1}\cdots z_t)\text{ (by Lemma \ref{lemmaDGInside})}\nonumber\\
=&(2^{t-j}-1)+2(2^{t-j}-1)(r-1)+r+d_{S(G,t-j)}((v_{r-1})^{t-j},z_{j+1}\cdots z_t)\nonumber\\
& \text{ (by Theorem \ref{theoExtremeVertices})}.\label{eqExtreme}
\end{align}
Now, let $P_2$ be a shortest path between $x^{t-j+1}$ and $z_j\cdots z_t$ and let $P_2'$ be the $G$-path associated to $P_2$. By Lemma \ref{lemmaDGExtremeVerticesFreeTriangle} and Theorem \ref{theoExtremeVertices}, we learned that $P_2'$ is triangle-free.
Suppose that $P_2'$ given by $x=u_0,u_1,u_2,\ldots, u_{s - 1},u_s=z_j$ has length  $s>d_G(x,z_j)$. Analogously to the way in which we obtained the length of $P_1$ in (\ref{eqExtreme}), we deduce that the length of $P_2$ is given by 
$$l(P_2)=(2^{t-j}-1)+2(2^{t-j}-1)(s-1)+s+d_{S(G,t-j)}((u_{s-1})^{t-j},z_{j+1}\cdots z_t).$$
Hence,
\begin{align*}
l(P_2)>&(2^{t-j}-1)+2(2^{t-j}-1)r+r+d_{S(G,t-j)}((u_{s-1})^{t-j},z_{j+1}\cdots z_t)\text{ (as }s\ge r+1)\\
=&(2^{t-j}-1)+2(2^{t-j}-1)(r-1)+2(2^{t-j}-1)+r+d_{S(G,t-j)}((u_{s-1})^{t-j},z_{j+1}\cdots z_t)\\
\ge &(2^{t-j}-1)+2(2^{t-j}-1)(r-1)+r+d_{S(G,t-j)}((v_{r-1})^{t-j},(u_{s-1})^{t-j})\\
&+d_{S(G,t-j)}((u_{s-1})^{t-j},z_{j+1}\cdots z_t)\text{ (by Theorem \ref{theoExtremeVertices} and }u_{s-1},v_{r-1}\in N(z_j))\\
\ge &(2^{t-j}-1)+2(2^{t-j}-1)(r-1)+r+
d_{S(G,t-j)}((v_{r-1})^{t-j},z_{j+1}\cdots z_t)\\
& \text{ (by triangle inequality)}\\
=&l(P_1),
\end{align*}
which is a contradiction. Therefore, the $G$-path associated to any shortest path between  $x^{t-j+1}$ and $z_j\cdots z_t$ is a shortest path between $x$ and $z_j$, so that (\ref{eqExtreme}) leads to the result.
\end{proof}

We can use Theorem \ref{remExtremeArbitraryVertices} as a tool to prove the following known result. 

\begin{corollary}{\rm \cite{Klavzar1997}}\label{remarkCompleteExtremeArbitraryVertices}
Let $t\ge 1$ and $n\ge 2$ be integers, let $K_n=(V,E)$ be a complete graph, $x\in V$ and and $w=z_1z_2\cdots z_t\in V^t$. Then 
$$d_{S(K_n,t)}(x^t,w)=\sum_{z_i\ne x} 2^{t-i}.$$
\end{corollary}

Theorem \ref{remExtremeArbitraryVertices} leads to Algorithm \ref{DG1algorithm} which allows us to compute the distance between an extreme vertex and any vertex of $S(G,t)$. In this algorithm we are using two functions, $dist(x,y)$ and $dist^*(x,y)$. The first one gives the distance between $x$ and $y$ and the second one gives the same distance and stores in $\nu(x,y)$ the set of neighbours of $y$ lying on the corresponding  shortest paths. 
For instance, if  we compute the distances by using Dijkstra's algorithm, then  Algorithm \ref{DG1algorithm} has time complexity of $O(t|V|^2)$, while Dijkstra's algorithm for $S(G,t)$ has time complexity of $O(|V|^{2t})$. 

\begin{algorithm}[h]
\caption{}
\label{DG1algorithm}
\begin{algorithmic}[0]
\State {\bf Input}: A connected graph $G$, a vertex $x$ and a word  $w = z_1z_{2}\cdots z_{t}$.
\State {\bf Output}:  $d_{S(G,t)}(x^t,w)$
\Function{RecursiveExtreme}{$x,w$}
\State $j\gets 1$
\While {$j< t\text{ \textbf{and} } x=z_j$}
\State $j\gets j+1$
\EndWhile
\If{$j=t$}
\State $d_{S(G,t)}(x^t,w)\gets dist(x,z_t)$
\Else
\State $m\gets +\infty$
\State $(d_G(x,z_j),\nu(x,z_j))\gets dist^*(x,z_j)$
\ForEach {$v \in \nu(x,z_j)$}
\State $m\gets \min\{m,\textsc{RecursiveExtreme}(v,z_{j+1}\cdots z_t) \}$
\EndFor
\State $d_{S(G,t)}(x^t,w)\gets m+(2^{t-j+1}-1)d_G(x,z_{j})-(2^{t-j}-1)$
\EndIf
\EndFunction
\end{algorithmic}
\end{algorithm}

%\section{Distances in bipartite graphs}
The result exposed in Lemma \ref{lemmaDGExtremeVerticesFreeTriangle} cannot be generalized to the case of two non-extreme vertices. For instance, for the graph $S(G,3)$ shown in Figure \ref{FigureExDistance} the $G$-path $1,2,4,3$ associated to the shortest path between $121$ and $344$ contains two triangles. Our next result concerns triangle-free $G$-paths.

\begin{theorem}\label{theoDistG2}
Let $G=(V,E)$ be a connected non-trivial graph, $t\ge 2$ an integer, $x,y,z\in V$ and $w,w'\in V^t$ such that $w=z^{j-1}xx_{j+1}\cdots x_t$, $w'=z^{j-1}yy_{j+1}\cdots y_t$, $1\le j\le t-1$ and $x\ne y$. If $(a)$ $x$ does not belong to any cycle or  $(b)$  any  path $x,u_1,\dots, u_{s-1},y$  of length  $s\ge d_G(x,y)+2$ satisfies that $x\not\in N_G(u_2)$ \text{\rm (}or $y\not\in N_G(u_{s-2}))$, then
$$d_{S(G,t)}(w,w')=\lambda(x,y)+(2^{t-j+1}-1)d_G(x,y)-2(2^{t-j}-1),$$
where $$\lambda(x,y)=\min_{P_i\in\mathcal{P}(x,y)}\left\{d_{S(G,t-j)}\left(\left(x^{(i)}\right)^{t-j},x_{j+1}\cdots x_t\right)+d_{S(G,t-j)}\left(\left(y^{(i)}\right)^{t-j},y_{j+1}\cdots y_t\right)\right\}.$$
\end{theorem}

\begin{proof}
By Lemma \ref{lemmaDGInside}, we have that $d_{S(G,t)}(w,w')=d_{S(G,t-j+1)}(xx_{j+1}\cdots x_t,yy_{j+1}\cdots y_t)$. Let $x=v_0,v_1,v_2,\ldots, v_{r - 1},v_r=y$ be a shortest path $P$ between $x$ and $y$. Let $P_1$ be a path of minimum length among all the paths from $xx_{j+1}\cdots x_t$ to $yy_{j+1}\cdots y_t$  having $P$ as its associated path. So, the length of $P_1$ is given by 
\begin{align}
l(P_1)=&d_{S(G,t-j+1)}(xx_{j+1}\cdots x_t,x(v_1)^{t-j})+\sum_{i=0}^{r-2} d_{S(G,t-j+1)}(v_{i+1}(v_i)^{t-j},v_{i+1}(v_{i+2})^{t-j})\nonumber\\
&+r+d_{S(G,t-j+1)}(v_r(v_{r-1})^{t-j},yy_{j+1}\cdots y_t)\nonumber\\
=&d_{S(G,t-j)}(x_{j+1}\cdots x_t,(v_1)^{t-j})+\sum_{i=0}^{r-2} d_{S(G,t-j)}((v_i)^{t-j},(v_{i+2})^{t-j})+r+\nonumber\\
&+d_{S(G,t-j)}((v_{r-1})^{t-j},y_{j+1}\cdots y_t)\text{ (by Lemma \ref{lemmaDGInside})}\nonumber\\
=&d_{S(G,t-j)}((x_{j+1}\cdots x_t,(v_1)^{t-j})+2(2^{t-j}-1)(r-1)+r\nonumber\\
&+d_{S(G,t-j)}((v_{r-1})^{t-j},(y_{j+1}\cdots y_t) \text{ (by Theorem \ref{theoExtremeVertices})}.\label{deductLength}
\end{align}
Let $P_2$ be a shortest path between $xx_{j+1}\cdots x_t$ and $yy_{j+1}\cdots y_t$, and let $x=u_0,u_1,u_2,$ $\ldots, u_{s - 1},u_s=y$ be the $G$-path $P_2'$ associated to $P_2$. Suppose that  $s>d_G(x,y)$.

We first assume premiss (a). In this case $u_1=v_1$ and  $s\ge r+1$, and  by Lemma \ref{lemmaDGExtremeVerticesFreeTriangle} and Theorem \ref{theoExtremeVertices} we can conclude that $P_2'$ is triangle-free.  Following a procedure analogous to that described in (\ref{deductLength}), we deduce that the length of $P_2$ is given by 
\begin{align*}
l(P_2)=&d_{S(G,t-j)}(x_{j+1}\cdots x_t,(v_1)^{t-j})+2(2^{t-j}-1)(s-1)+s\\
&+d_{S(G,t-j)}((u_{s-1})^{t-j},y_{j+1}\cdots y_t).
\end{align*}
Hence,
\begin{align*}
l(P_2)
>&d_{S(G,t-j)}(x_{j+1}\cdots x_t,(v_1)^{t-j})+2(2^{t-j}-1)r+r\\
&+d_{S(G,t-j)}((u_{s-1})^{t-j},y_{j+1}\cdots y_t)\text{ (as }s\ge r+1)\\
=&d_{S(G,t-j)}(x_{j+1}\cdots x_t,(v_1)^{t-j})+2(2^{t-j}-1)(r-1)+r\\
&+2(2^{t-j}-1)+d_{S(G,t-j)}((u_{s-1})^{t-j},y_{j+1}\cdots y_t)\\
\ge &d_{S(G,t-j)}(x_{j+1}\cdots x_t,(v_1)^{t-j})+2(2^{t-j}-1)(r-1)+r\\
&+d_{S(G,t-j)}((v_{r-1})^{t-j},(u_{s-1})^{t-j})+d_{S(G,t-j)}((u_{s-1})^{t-j},y_{j+1}\cdots y_t)\\
&\text{ (by Theorem \ref{theoExtremeVertices} and the fact that  $u_{s-1},v_{r-1}\in N(y)$)}\\
\ge &d_{S(G,t-j)}((x_{j+1}\cdots x_t,(v_1)^{t-j})+2(2^{t-j}-1)(r-1)+r\\
&+d_{S(G,t-j)}((v_{r-1})^{t-j},y_{j+1}\cdots y_t) \text{ (by triangle inequality)}\\
=&l(P_1),
\end{align*}
which is a contradiction.

%%%%%%%%%%%% (c)
Finally, assume  premiss (b). In this case, $s\ge r+2$ and by  Lemma \ref{lemmaDGExtremeVerticesFreeTriangle} and Theorem \ref{theoExtremeVertices} we can conclude that $P_2'$ is triangle-free. Following a procedure analogous to that described in (\ref{deductLength}), we deduce that the length of $P_2$ is given by 
\begin{align*}
l(P_2)=&d_{S(G,t-j)}(x_{j+1}\cdots x_t,(u_1)^{t-j})+2(2^{t-j}-1)(s-1)+s\\
&+d_{S(G,t-j)}((u_{s-1})^{t-j},y_{j+1}\cdots y_t).
\end{align*}
Hence,
\begin{align*}
l(P_2)
>&d_{S(G,t-j)}(x_{j+1}\cdots x_t,(u_1)^{t-j})+2(2^{t-j}-1)(r+1)+r\\
&+d_{S(G,t-j)}((u_{s-1})^{t-j},y_{j+1}\cdots y_t)\text{ (as }s\ge r+2)\\
=&d_{S(G,t-j)}(x_{j+1}\cdots x_t,(u_1)^{t-j})+2(2^{t-j}-1)+2(2^{t-j}-1)(r-1)+r\\
&+2(2^{t-j}-1)+d_{S(G,t-j)}((u_{s-1})^{t-j},y_{j+1}\cdots y_t)\\
\ge &d_{S(G,t-j)}(x_{j+1}\cdots x_t,(u_1)^{t-j})+d_{S(G,t-j)}((u_1)^{t-j},(v_1)^{t-j})+2(2^{t-j}-1)(r-1)\\
&+r+d_{S(G,t-j)}((v_{r-1})^{t-j},(u_{s-1})^{t-j})+d_{S(G,t-j)}((u_{s-1})^{t-j},y_{j+1}\cdots y_t)\\
&\text{ (by Theorem \ref{theoExtremeVertices} and the fact that $u_{1},v_{1}\in N(x)$ and $u_{s-1},v_{r-1}\in N(y)$)}\\
\ge &d_{S(G,t-j)}((x_{j+1}\cdots x_t,(v_1)^{t-j})+2(2^{t-j}-1)(r-1)+r\\
&+d_{S(G,t-j)}((v_{r-1})^{t-j},y_{j+1}\cdots y_t) \text{ (by triangle inequality)}\\
=&l(P_1),
\end{align*}
which is a contradiction.

Therefore, the $G$-path associated to any shortest path between  $xx_{j+1}\cdots x_t$ and $yy_{j+1}\cdots y_t$ is a shortest path between $x$ and $y$, so that
 (\ref{deductLength}) leads to the result.
\end{proof}

From Theorem \ref{theoDistG2} we can state the formula for the distance between vertices in $S(G,t)$ for any bipartite graph $G$.

\begin{corollary}\label{bipartiteDist}
Let $G=(V,E)$ be a connected bipartite graph, $t\ge 2$ an integer, $x,y,z\in V$ and $w,w'\in V^t$ such that $w=z^{j-1}xx_{j+1}\cdots x_t$, $w'=z^{j-1}yy_{j+1}\cdots y_t$, $1\le j\le t-1$ and $x\ne y$.  Then
$$d_{S(G,t)}(w,w')=\lambda(x,y)+(2^{t-j+1}-1)d_G(x,y)-2(2^{t-j}-1),$$
where $$\lambda(x,y)=\min_{P_i\in\mathcal{P}(x,y)}\left\{d_{S(G,t-j)}\left(\left(x^{(i)}\right)^{t-j},x_{j+1}\cdots x_t\right)+d_{S(G,t-j)}\left(\left(y^{(i)}\right)^{t-j},y_{j+1}\cdots y_t\right)\right\}.$$
\end{corollary}

For instance, for the cycle graph $C_4=(V,E)$ whose vertex set is $V=\{a,b,c,d\}$ and edge set is $E=\{\{a,b\},\{b,c\},\{c,d\},\{d,a\}\}$, Corollary \ref{bipartiteDist} leads to $d_{S(C_4,3)}(dab,bdc)=13$. In this case, $t=3$, $j=1$, $x=d$, $y=b$ and 
\begin{align*}
\lambda(d,b)&=\min\{d_{S(C_4,2)}(aa,ab)+d_{S(C_4,2)}(aa,dc),d_{S(C_4,2)}(cc,ab)+d_{S(C_4,2)}(cc,dc)\}\\
&=\min\{1+4,5+2\}\\
&=5.
\end{align*}
Analogously, $d_{S(C_4,3)}(dab,cad)=8$, as $t=3$, $j=1$, $x=d$, $y=c$ and 
$
\lambda(d,c)=d_{S(C_4,2)}(cc,ab)+d_{S(C_4,2)}(dd,ad)=5+2=7.
$

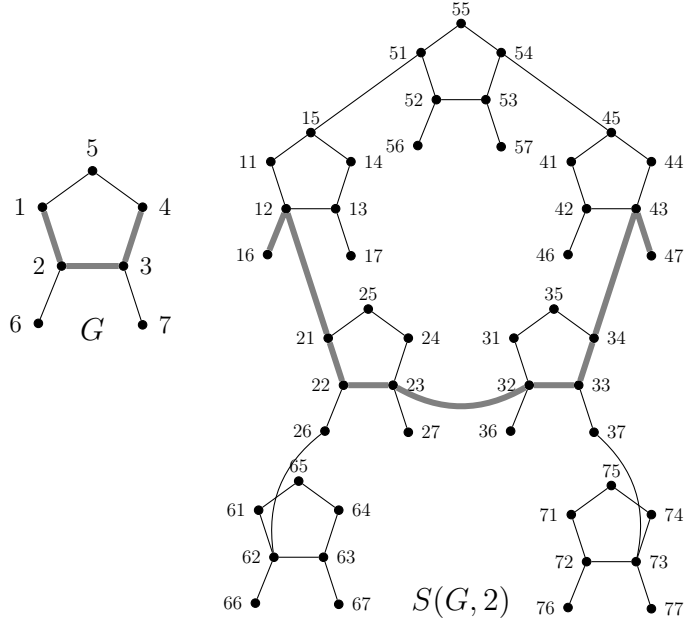
\begin{figure}[!ht]
\centering
\begin{tikzpicture}[transform shape, inner sep = .4mm]
%Graph G
\def\radius{0.7} %Radius of cycle 
\foreach \ind in {1,...,5}
{
\pgfmathparse{90 + 360/5*\ind};
\node [draw=black, shape=circle, fill=black] (\ind) at (\pgfmathresult:\radius cm) {};
\ifthenelse{\ind=5}
{
\node [scale=.8] at ([yshift=.3 cm]\ind) {$\ind$};
}
{
\ifthenelse{\ind=3 \OR \ind=4}
{
\node [scale=.8] at ([xshift=.3 cm]\ind) {$\ind$};
}
{
\node [scale=.8] at ([xshift=-.3 cm]\ind) {$\ind$};
};
};
\ifthenelse{\ind>1}
{
\pgfmathparse{int(\ind-1)};
\draw[black] (\ind) -- (\pgfmathresult);
}
{};
}
\draw[black] (1) -- (5);
\pgfmathparse{\radius*sin(72)/sin(54)};
\node [draw=black, shape=circle, fill=black] (6) at ([shift=({248:\pgfmathresult cm})]2) {};
\node [scale=.8] at ([xshift=-.3 cm]6) {$6$};
\draw[black] (2) -- (6);
\node [draw=black, shape=circle, fill=black] (7) at ([shift=({-72:\pgfmathresult cm})]3) {};
\node [scale=.8] at ([xshift=.3 cm]7) {$7$};
\draw[black] (3) -- (7);
\node at ([yshift=-3*\radius cm]5) {$G$};
%%%%%%%%%%%%%%%%%%%%%%%%%%%%%%%%%%%%%%%%%%%%%%%
%S(G,2)
\def\widenodetwo{.4};
\pgfmathparse{7*\radius};
\node (center) at (\pgfmathresult cm,0) {};
%%%Copies of G corresponding to vertices 1,2,3,4,5
\foreach \x in {1,...,5}
{
\pgfmathparse{90 + 360/5*\x};
\node (center\x) at ([shift=({\pgfmathresult:3*\radius cm})]center) {};
\foreach \ind in {1,...,5}
{
\pgfmathparse{90 + 360/5*\ind};
\node [draw=black, shape=circle, fill=black, inner sep = \widenodetwo mm] (\x\ind) at ([shift=({\pgfmathresult:.8*\radius cm})]center\x) {};
\ifthenelse{\ind=5}
{
\node [scale=.6] at ([yshift=.2 cm]\x\ind) {$\x\ind$};
}
{
\ifthenelse{\ind=3 \OR \ind=4}
{
\node [scale=.6] at ([xshift=.3 cm]\x\ind) {$\x\ind$};
}
{
\node [scale=.6] at ([xshift=-.3 cm]\x\ind) {$\x\ind$};
};
};
\ifthenelse{\ind>1}
{
\pgfmathparse{int(\ind-1)};
\draw[black] (\x\ind) -- (\x\pgfmathresult);
}
{};
}
\draw[black] (\x1) -- (\x5);
\pgfmathparse{.8*\radius*sin(72)/sin(54)};
\node [draw=black, shape=circle, fill=black, inner sep = \widenodetwo mm] (\x6) at ([shift=({248:\pgfmathresult cm})]\x2) {};
\node [scale=.6] at ([xshift=-.3 cm]\x6) {$\x6$};
\draw[black] (\x2) -- (\x6);
\node [draw=black, shape=circle, fill=black, inner sep = \widenodetwo mm] (\x7) at ([shift=({-72:\pgfmathresult cm})]\x3) {};
\node [scale=.6] at ([xshift=.3 cm]\x7) {$\x7$};
\draw[black] (\x3) -- (\x7);
}
%%%Copies of G corresponding to vertices 6,7
\pgfmathparse{3*\radius*sin(72)/sin(54)};
\node (center6) at ([shift=({248:\pgfmathresult cm})]center2) {};
\node (center7) at ([shift=({-72:\pgfmathresult cm})]center3) {};
\foreach \x in {6,7}
{
\foreach \ind in {1,...,5}
{
\pgfmathparse{90 + 360/5*\ind};
\node [draw=black, shape=circle, fill=black, inner sep = \widenodetwo mm] (\x\ind) at ([shift=({\pgfmathresult:.8*\radius cm})]center\x) {};
\ifthenelse{\ind=5}
{
\node [scale=.6] at ([yshift=.2 cm]\x\ind) {$\x\ind$};
}
{
\ifthenelse{\ind=3 \OR \ind=4}
{
\node [scale=.6] at ([xshift=.3 cm]\x\ind) {$\x\ind$};
}
{
\node [scale=.6] at ([xshift=-.3 cm]\x\ind) {$\x\ind$};
};
};
\ifthenelse{\ind>1}
{
\pgfmathparse{int(\ind-1)};
\draw[black] (\x\ind) -- (\x\pgfmathresult);
}
{};
}
\draw[black] (\x1) -- (\x5);
\pgfmathparse{.8*\radius*sin(72)/sin(54)};
\node [draw=black, shape=circle, fill=black, inner sep = \widenodetwo mm] (\x6) at ([shift=({248:\pgfmathresult cm})]\x2) {};
\node [scale=.6] at ([xshift=-.3 cm]\x6) {$\x6$};
\draw[black] (\x2) -- (\x6);
\node [draw=black, shape=circle, fill=black, inner sep = \widenodetwo mm] (\x7) at ([shift=({-72:\pgfmathresult cm})]\x3) {};
\node [scale=.6] at ([xshift=.3 cm]\x7) {$\x7$};
\draw[black] (\x3) -- (\x7);
}
%Edges between copies of G
\foreach \u/\v in {1/2,2/3,3/4,4/5,5/1,2/6,3/7}
{
\ifthenelse{\u=2 \AND \v=3 \OR \v=6}
{
\draw (\u\v) to[bend right] (\v\u);
}
{
\ifthenelse{\u=3 \AND \v=7}
{
\draw (\u\v) to[bend left] (\v\u);
}
{
\draw[black] (\u\v) -- (\v\u);
};
};
}
\node at ([yshift=-11*\radius cm]55) {$S(G,2)$};
\draw[gray, line width=.8 mm] (16)--(12)--(21)--(22)--(23);
\draw[gray, line width=.8 mm]  (32)--(33)--(34)--(43)--(47);
\draw[gray, line width=.8 mm] (23) to[bend right] (32);
\draw[gray, line width=.8 mm] (1)--(2)--(3)--(4);
\end{tikzpicture}
\caption{The $G$-path $1,2,3,4$ associated to the shortest path $16,12,21,22,23,$ $32,33,34,43,47$ between $16$ and $47$, is not a shortest path between the vertices $1$ and $4$.}\label{CounterExample}
\end{figure}

Given two vertices $x,y\in V$, we define $\mathcal{P}'(x,y)$ as the set of all paths between $x$ and $y$ of length $d_G(x,y)+1$. For any $P_k\in \mathcal{P}'(x,y)$ between $x$ and $y$, the neighbour of $y$ lying on $P_k$ will be denoted by $y^{(k)}$. With this notation in mind we proceed to state the following  result which can be deduced by analogy to the proof of  Theorem \ref{theoDistG2}.

\begin{theorem}\label{theoDistGTriangleFree}
Let $G=(V,E)$ be a connected non-trivial triangle-free graph, $t\ge 2$ an integer, $x,y,z\in V$ and $w,w'\in V^t$ such that $w=z^{j-1}xx_{j+1}\cdots x_t$, $w'=z^{j-1}yy_{j+1}\cdots y_t$, $1\le j\le t-1$ and $x\ne y$. Then 
$$
d_{S(G,t)}(w,w')=\min\{\vartheta(x,y),\vartheta'(x,y)\},
$$
where 
$$\vartheta(x,y)=\lambda(x,y)+\left(2^{t-j+1}-1\right) d_G(x,y)-2(2^{t-j}-1),$$
$$\vartheta'(x,y)=\lambda'(x,y)+\left(2^{t-j+1}-1\right)d_G(x,y)+1,$$
$$\lambda(x,y)=\min_{P_i\in\mathcal{P}(x,y)}\left\{d_{S(G,t-j)}\left(\left(x^{(i)}\right)^{t-j},x_{j+1}\cdots x_t\right)+d_{S(G,t-j)}\left(\left(y^{(i)}\right)^{t-j},y_{j+1}\cdots y_t\right)\right\},$$
and
$$\lambda'(x,y)=\min_{P_k\in\mathcal{P}'(x,y)}\left\{d_{S(G,t-j)}\left(\left(x^{(k)}\right)^{t-j},x_{j+1}\cdots x_t\right)+d_{S(G,t-j)}\left(\left(y^{(k)}\right)^{t-j},y_{j+1}\cdots y_t\right)\right\}.$$
\end{theorem}

Figure \ref{CounterExample} shows an example where the path associated to the shortest path between the vertices $16$ and $47$ of $S(G,2)$  is not a shortest path between $1$ and $4$. According to Theorem \ref{theoDistGTriangleFree} we have that $d_{S(G,2)}(14,47)=9=\vartheta'(1,4)$, as $\lambda'(1,4)=2$ while $\lambda(1,4)=6$.

Theorem \ref{theoDistGTriangleFree} leads to Algorithm \ref{DG2algorithm} which allows us to compute the distance between two arbitrary vertices of $S(G,t)$ for any connected triangle-free graph $G$. In this algorithm we are using a function, $dist^{**}(x,y)$ which gives the distance between $x$ and $y$ and stores in $\varphi(x,y)$ the set of pairs $(x',y')$ such that $x'$ and $y'$ are, respectively, neighbours of $x$ and $y$ lying on the  shortest paths between them, and stores in $\varphi'(x,y)$ the set of pairs $(x'',y'')$ such that $x''$ and $y''$ are, respectively, neighbours of $x$ and $y$ lying on the   paths of length $d_G(x,y)+1$. 

\begin{algorithm}[h]
\caption{}
\label{DG2algorithm}
\begin{algorithmic}[0]
\State {\bf Input}: A connected triangle-free graph $G$  and two words $w = x_1x_{2}\cdots x_{t}$ and $w = y_1y_{2}\cdots y_{t}$.
\State {\bf Output}:  $d_{S(G,t)}(w,w')$
\State $j\gets 1$
\While {$j< t\text{ \textbf{and} } x_j=y_j$}
\State $j\gets j+1$
\EndWhile
\If{$j=t$}
\State $d_{S(G,t)}(w,w')\gets dist(x_t,y_t)$
\Else
\State $(d_G(x_j,y_j),\varphi(x_j,y_j),\varphi'(x_j,y_j))\gets dist^{**}(x_j,y_j)$
\State $\lambda\gets +\infty$
\ForEach {$(u,v) \in \varphi(x_j,y_j)$}
\State $d_x \gets \textsc{RecursiveExtreme}(u,x_{j+1}\cdots x_t)$\State $d_y \gets \textsc{RecursiveExtreme}(v,y_{j+1}\cdots y_t)$
\State $\lambda\gets \min\{\lambda, d_x+d_y\}$
\EndFor
\State $\lambda'\gets +\infty$
\ForEach {$(u,v) \in \varphi'(x_j,y_j)$}
\State $d_x \gets \textsc{RecursiveExtreme}(u,x_{j+1}\cdots x_t)$\State $d_y \gets \textsc{RecursiveExtreme}(v,y_{j+1}\cdots y_t)$
\State $\lambda'\gets \min\{\lambda', d_x+d_y\}$
\EndFor
\State $\vartheta\gets \lambda+(2^{t-j+1}-1)d_G(x,z_j)-2(2^{t-j}-1)$
\State $\vartheta'\gets \lambda'+(2^{t-j+1}-1)d_G(x,z_j)+1$
\State $d_{S(G,t)}(w,w')\gets\min\{\vartheta, \vartheta'\}$
\EndIf
\end{algorithmic}
\end{algorithm}

\section{Distances in trees}
As shown in \cite{Rodriguez-Velazquez2015a}, for any tree $T$ and any positive integer $t$ the Sierpi\'{n}ski graph $S(T,t)$ is a tree. Thus, there exists only one path between two vertices of $S(T,t)$. 
Notice that Corollary \ref{bipartiteDist} leads to the next remark.

\begin{remark}\label{remTreeDistG}
Let $T=(V,E)$ be a tree and  $t\ge 2$ an integer. For any  $w=z^{j-1}xx_{j+1}\cdots x_t$ and $w'=z^{j-1}yy_{j+1}\cdots y_t$, where $1\le j\le t-1$,  $x,y,z\in V$, $x\ne y$, and $w,w'\in V^t$,
\begin{align*}
d_{S(T,t)}(w,w')=&d_{S(T,t-j)}\left(\left(x'\right)^{t-j},x_{j+1}\cdots x_t\right)+d_{S(T,t-j)}\left(\left(y'\right)^{t-j},y_{j+1}\cdots y_t \right)\\
&+(2^{t-j+1}-1)d_T(x,y)-2(2^{t-j}-1),
\end{align*}
where $x'$ and $y'$ are the neighbours of $x$ and $y$ lying on the path between $x$ and $y$, respectively.
\end{remark}

%\subsection{Eccentricity, Diameter and Radius}
The \textit{eccentricity} $\epsilon(v)$ of a vertex $v$ in a connected graph $G$ is the maximum distance between $v$ and any other vertex $u$ of $G$. The \textit{diameter} of $G$ is defined as $$D(G)=\displaystyle\max_{v\in V(G)}\{\epsilon(v)\},$$
and the \textit{radius} of $G$ is defined as $$r(G)=\displaystyle\min_{v\in V(G)}\{\epsilon(v)\}.$$ For a vertex $v$, each vertex at distance $\epsilon(v)$ from $v$ is an \emph{eccentric vertex} for $v$. A \emph{leaf} in a tree is a vertex of degree one, while a \emph{support vertex} is a vertex adjacent to a leaf.

\begin{remark}\label{LemmaIsLeaf}
Let $u$ and $v$ be two different vertices in a tree $T$. If   $v$ is an eccentric vertex for $u$, then $v$ is a leaf and $\epsilon(v)=D(T)$.
\end{remark}

From now on we will assume that $T$ has order $n\ge 3$, as $S(K_2,t)\cong P_{2^t}$.

\begin{lemma}\label{lemmaFormulaEccTreesTemp}
Let $T=(V,E)$ be a tree of order greater than or equal to three, $u,v\in V$ and  $t\ge 2$ an integer. Then the following statements hold.
\begin{enumerate}[{\rm (i)}]
\item If $\epsilon(u)\ge \epsilon(v)$, then $\epsilon(u^t)\ge \epsilon(v^t)$.
\item 
$\epsilon(u^t)= \epsilon\left(z^{t-1}\right)+(2^t - 1)\epsilon(u)-(2^{t-1}-1) ,$
where $z$ is the support vertex of an eccentric vertex for $u$.
\end{enumerate}
\end{lemma}

\begin{proof}
Let $w=u^{i-1}xx_{i+1}\cdots x_t\in V^t$ and $w'=v^{j-1}yy_{j+1}\cdots y_t\in V^t$ such that $\epsilon(u^t)=d(u^t,w)$ and $\epsilon(v^t)=d(v^t,w')$. 
If $i>1$, then for any $u'\in N_T(u)$ we have that   
\begin{align*}
d_{S(T,t)}(u^t,u'u^{i-2}xx_{i+1}\cdots x_t)&>d_{S(T,t)}(u^{t},u(u')^{t-1})+d_{S(T,t)}(u'u^{t-1},u'u^{i-2}xx_{i+1}\cdots x_t)\\
&>d_{S(T,t)}(u'u^{t-1},u'u^{i-2}xx_{i+1}\cdots x_t)\\
&=d_{S(T,t)}(u^t,w),
\end{align*}
 which is a contradiction. Hence, $i=j=1$ and by Remark \ref{remTreeDistG} we have
\begin{equation}\label{EQEccExtTree1}
d_{S(T,t)}(u^t,w)=d_{S(T,t-1)}(x'\cdots x',x_{2}\cdots x_t)+(2^{t}-1)d_T(u,x)-(2^{t-1}-1)
\end{equation}
and
\begin{equation}
d_{S(T,t)}(v^t,w')=d_{S(T,t-1)}(y'\cdots y',y_{2}\cdots y_t)+(2^{t}-1)d_T(v,y)-(2^{t-1}-1),
\end{equation} 
where $x'$ is the neighbour of $x$ lying on the path between $x$ and $u$ and  $y'$ is the neighbour of $y$ lying on the path between $y$ and $v$.
From now on we assume that $\epsilon(u)\ge \epsilon(v)$ and then we will show that  $\epsilon(u^t)\ge \epsilon(v^t)$ by induction. Let  $t=2$. By \eqref{EQEccExtTree1}
$d_{S(T,2)}(u^2,w)=d_{T}(x',x_{2})+3d_T(u,x)-1\le (D(T)-1)+3\epsilon(u)-1$ and the equality holds for $x,x_2$ satisfying $d_T(x,u)=\epsilon(u)$ and $d_T(x_2,x)=D(T)$. Hence, 
$$
 \epsilon(u^2)=D(T)+3\epsilon(u)-2\ge D(T)+3\epsilon(v)-2=\epsilon(v^2).
$$
Our hypothesis is that $\epsilon(u^t)\ge \epsilon(v^t)$. For $x\in V$ such that $d_T(x,u)=\epsilon(u)$,  and taking $x'$ as the neighbour of $x$ lying on the path between $x$ and $u$, we have $\epsilon(x')=D(T)-1$ (by Remark \ref{LemmaIsLeaf}). Hence,    by hypothesis we have that $\epsilon((x')^t)\ge \epsilon((y')^t)$, for every internal vertex $y'$. Thus, \eqref{EQEccExtTree1} leads to
\begin{equation} \label{EQ-FormulaEccExtTree}
\epsilon(u^{t+1})=\epsilon((x')^t)+(2^{t+1}-1)\epsilon(u)-(2^{t}-1) \end{equation}
and, analogously,
\begin{equation}
\epsilon(v^{t+1})=\epsilon((y')^t)+(2^{t+1}-1)\epsilon(v)-(2^{t}-1), \end{equation}
which implies that $\epsilon(u^{t+1})\ge \epsilon(v^{t+1})$. Therefore, (i) follows by induction and (ii) by \eqref{EQ-FormulaEccExtTree}. 
\end{proof}

\begin{theorem}\label{EccentricityExtremeFormula}
Let $T=(V,E)$ be a tree of order greater than or equal to three. Then for any $u\in V$ and any integer $t\ge 1$, $$\epsilon(u^t)=(2^t-1)\epsilon(u)+(2^t-t-1)(D(T)-2).$$
\end{theorem}

\begin{proof}
We proceed by induction on $t$. The equality holds for $t=1$. Suppose that  $\epsilon(u^t)=(2^t-1)\epsilon(u)+(2^t-t-1)(D(T)-2).$
Then we have
\begin{align*}
\epsilon(u^{t+1})=&\epsilon(v^t)+(2^{t+1}-1)\epsilon(u)-(2^t-1) \text{ (by  Lemma \ref{lemmaFormulaEccTreesTemp} (ii))}\\
=&(2^t-1)\epsilon(v)+(2^t-t-1)(D(T)-2)+(2^{t+1}-1)\epsilon(u)-(2^t-1)\\
&\text{  (by hipothesis)}\\
=&(2^t-1)(D(T)-1)+(2^t-t-1)(D(T)-2)+(2^{t+1}-1)\epsilon(u)-(2^t-1)\\
&\text{  (as $v$ is the support of a diametral vertex)}\\ 
=&(2^{t+1}-1)\epsilon(u)+(2^{t+1}-(t+1)-1)(D(T)-2).
\end{align*}
Therefore, the result follows
\end{proof}

\begin{theorem}\label{DiameterSierpinskiTrees}
For any tree $T$ of order greater than or equal to three and any positive integer $t$,
$$D(S(T,t))=(3\cdot 2^t-2t-3)D(T)-4(2^t-t-1).$$
\end{theorem}

\begin{proof}
Let  $u,u',v,v'\in V$ and $w_1,w_2\in V^{t-1}$ such that $d_T(u,v)=D(T)$, $u'$ and $v'$ are the support vertices of $u$ and $v$, respectively,  $d_{S(T,t-1)}((u')^{t-1},w_1)=\epsilon((u')^{t-1})$ and $d_{S(T,t-1)}((v')^{t-1},w_2)=\epsilon((v')^{t-1})$. By Remark \ref{remTreeDistG} and Lemma \ref{lemmaFormulaEccTreesTemp} (i) we have that for any $x,y\in V$ and $w,w'\in V^{t-1}$,
$$d_{S(T,t)}(uw_1,vw_2)=\epsilon((u')^{t-1})+\epsilon((v')^{t-1})+(2^t-1)D(T)-2(2^{t-1}-1)\ge d_{S(T,t)}(xw,yw').$$
Therefore, $uw_1$ and $ vw_2$ are diametral vertices and so Theorem \ref{EccentricityExtremeFormula} leads to 
$$
D(S(T,t))=d_{S(T,t)}(uw_1,vw_2)=(3\cdot 2^t-2t-3)D(T)-4(2^t-t-1).
$$
\end{proof}

It is well known that if $D(T)$ is even, then $r(T)=\frac{1}{2}D(T)$, otherwise, $r(T)=\frac{1}{2}(D(T)+1)$. Now, by Theorem \ref{DiameterSierpinskiTrees} we have that $D(T)$ is even if and only if $D(S(T,t))$ is even, so that we deduce the following result.

\begin{theorem}
For any tree $T$ of order greater than or equal to three and any positive integer $t$,
\[
r(S(T, t))= \left\{ \begin{array}{ll}
\dfrac{1}{2}(3 \cdot 2^t - 2t - 3)D(T)-2(2^t - t - 1), & \text{for } D(T)\text{ even;}\\
&\\
\dfrac{1}{2}\left( (3 \cdot 2^t - 2t - 3)D(T) - 2^{t + 2} + 4t +5\right), & \text{otherwise.}
\end{array}
\right.
\]
\end{theorem}

\end{document}